\documentclass[11pt, letter]{article} 

\usepackage{amscd, amsfonts, amsmath, amssymb, amsthm, bm, epsf, epsfig, flafter, float, subfigure, comment, marvosym}
\usepackage[margin=1in]{geometry}

\theoremstyle{plain}
\newtheorem{thm}{Theorem}
\newtheorem{lemma}[thm]{Lemma}

\newtheorem*{claim*}{Claim}
\newtheorem{prop}[thm]{Proposition}
\newtheorem{cor}[thm]{Corollary}

\theoremstyle{definition}

\theoremstyle{remark}

\makeatletter
\newcommand*{\rom}[1]{\expandafter\@slowromancap\romannumeral #1@}
\makeatother

\newcommand{\bbn}{\mathbb{N}}
\newcommand{\bbr}{\mathbb{R}}
\newcommand{\cals}{\mathcal{S}}
\newcommand{\calt}{\mathcal{T}}
\newcommand{\fraks}{\mathfrak{s}}
\newcommand{\frakt}{{\mathfrak{t}}}

\newcommand\abs[1]{\left|{#1}\right|}

\begin{document}

\begin{center}

{\bf\large{A geometric iterated function system on triangles}}

\vskip 0.5cm 

{Jiajun Wang}
\vskip 0.1in 
{\small{ \textit{LMAM, School of Mathematics Sciences, Peking University\\Beijing, 100871, P. R. China\\ wjiajun@pku.edu.cn}}}

\vskip 0.2in 

{Ying Zhang}
\vskip 0.1in 
{\small{ \textit{School of Mathematics Sciences, Soochow University\\Suzhou, 215006, P. R. China\\yzhang@suda.edu.cn}}}

\end{center}

\vskip 0.2in
\baselineskip 16pt
\abstract{We construct a converging geometric iterated function system on the moduli space of ordered triangles, for which the involved functions have geometric meanings and contain a non-contraction map under the natural metric.}
\date{September, 2011}

\section{Introduction}\label{sec:intro}

Iterated function systems are often used to construct fractals. Given a metric space, a classical \emph{iterated function system} is a finite set of functions
$$\left\{f_i:X\rightarrow X\ \big|\ i=1,2,\dots,N\right\}\quad (N\in\bbn)$$
such that each $f_i$ is a contraction. In \cite{Barnsley:2011dy}, Barnsley and Vince showed that the IFSs of noncontractive type (i.e. composed of maps that are not contractions with respect to any topologically equivalent metric in X) can yield attractor. These arise naturally in projective spaces, though classical irrational rotation on the circle can be adapted too.

In the present paper, we will construct a converging geometric IFS for which the functions are not contraction maps under the natural metric.

We study the limit behavior of iteratedly dividing a triangle in a natural way. Let $\calt$ be the moduli space of hyperbolic/Euclidean triangles modulo (ordered) similarities. A hyperbolic triangle is given in Figure \ref{fig:hyperbolic_triangle}.
\begin{figure}[htbp]
\begin{center}
\includegraphics{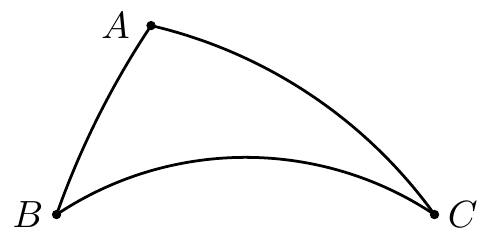}
\end{center}
\begin{minipage}{0.9\textwidth}
\caption{\label{fig:hyperbolic_triangle}
{\bf A hyperbolic triangle in the upper plane model.}}
\end{minipage}
\end{figure}
An ordered triangle is parametrized by the (ordered) triple of its three angles. $\calt$ is naturally identified with the tetrahedron $\left\{(x,y,z)\in\bbr^3\ \big|\ x,y,z>0, x+y+z\leqslant\pi\right\}$ in $\bbr^3$. Given a hyperbolic triangle $\triangle_{ABC}$,  the three \emph{mid-lines} joining the midpoints of its edges divide it into four smaller triangles.
\begin{figure}[htp]
\center{\includegraphics[width=400pt]{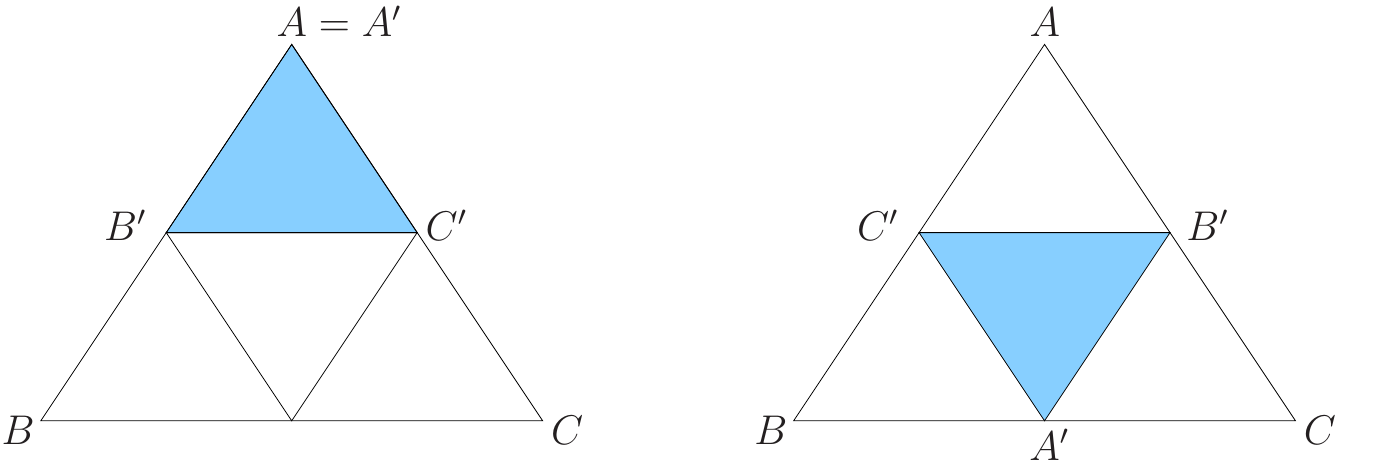}}\\
\begin{minipage}{0.9\textwidth}
\caption{\label{fig:defining_triangle}
{\bf The defining functions} The shadowed triangles are $f_A(\triangle_{ABC})$ and $f_M(\triangle_{ABC})$ respectively. }
\end{minipage}
\end{figure}
Let $f_A, f_B, f_C$ and $f_M$ be functions on $\calt$ that maps $\triangle_{ABC}$ to one of the smaller triangles, with vertices ordered in the natural way so that they are identical for Euclidean ones, as indicated in Figure \ref{fig:defining_triangle}. Our main result is the following

\begin{thm}\label{thm:main}
The iterated function system $\{f_A,f_B,f_C,f_M\}$ on $\calt$ is converging and the limit functions are continuous.
\end{thm}

The functions can be determined as follows. Let $a, b, c$ be the length of edges $BC$, $CA$ and $AB$ respectively and define $a^\prime$, $b^\prime$ and $c^\prime$ to be the length of $f_M(\triangle_{ABC})$ in the similar way. Then $a^\prime, b^\prime$ and $c^\prime$ are related to $A$, $B$ and $C$ by
\begin{align}\label{eqn:edge_angle_relation}
\begin{split}
\cos(A)=\frac{\cosh(b^\prime)\cosh(c^\prime)-\cosh(a^\prime)}{\sinh(b^\prime)\sinh(c^\prime)},\\
\cos(B)=\frac{\cosh(c^\prime)\cosh(a^\prime)-\cosh(b^\prime)}{\sinh(c^\prime)\sinh(a^\prime)},\\
\cos(C)=\frac{\cosh(a^\prime)\cosh(b^\prime)-\cosh(c^\prime)}{\sinh(a^\prime)\sinh(b^\prime)}.
\end{split}
\end{align}
And the edge lengths $a$, $b$ and $c$ are related to $a^\prime$, $b^\prime$ and $c^\prime$ by
\begin{equation}\label{eqn:edge_length_iteration}
\cosh(a^\prime) =\cosh\left(\frac{a}{2}\right)\cdot\mu,\quad \cosh(b^\prime) =\cosh\left(\frac{b}{2}\right)\cdot\mu,\quad \cosh(c^\prime) =\cosh\left(\frac{c}{2}\right)\cdot\mu,
\end{equation}
where
$$\mu:=\frac{1-\displaystyle{\tanh\frac{a+b+c}{4} \tanh\frac{a+b-c}{4} \tanh\frac{c+a-b}{4} \tanh\frac{b+c-a}{4}}}
{1+\displaystyle{\tanh\frac{a+b+c}{4} \tanh\frac{a+b-c}{4} \tanh\frac{c+a-b}{4} \tanh\frac{b+c-a}{4}}}.$$
We use Equation \ref{eqn:edge_angle_relation} to compute the edge lengths of $f_M(\triangle_{ABC})$. Then the edge lengths of the image triangle $\triangle_*$ of $\triangle_{ABC}$ under $f_A$, $f_B$, $f_C$ or $f_M$ is known.  Now Equation \ref{eqn:edge_length_iteration} determines the edge lengths of the triangle $f_M(\triangle_*)$ and Equation \ref{eqn:edge_angle_relation} then compute the angles of $\triangle_*$ from the edge lengths of $f_M(\triangle_*)$.

We remark that the corresponding result holds equally well for spherical triangles. We may also ask the following questions. Are the limit functions smooth? Does the limit function give a half-line bundle structure over $\calt$?

The papers is organized as follows. In Section 2, we study the iterated behaviours of the edge lengths, areas and angles. In Section 3, we show the continuity of the limit functions.

\section{Iterated division}

We refer to \cite{Casey:1889} for formulas in hyperbolic geometry.

Let $\{\triangle_n\}_{n=0}^\infty$ be a sequence of iterated action of the functions $f_A,f_B,f_C$ and $f_M$ on the triangle $\triangle_0=\triangle_{A_0B_0C_0}$, as illustrated in Figure \ref{fig:iteration_triangle}. 
\begin{figure}[htbp]
\center{\includegraphics[width=400pt]{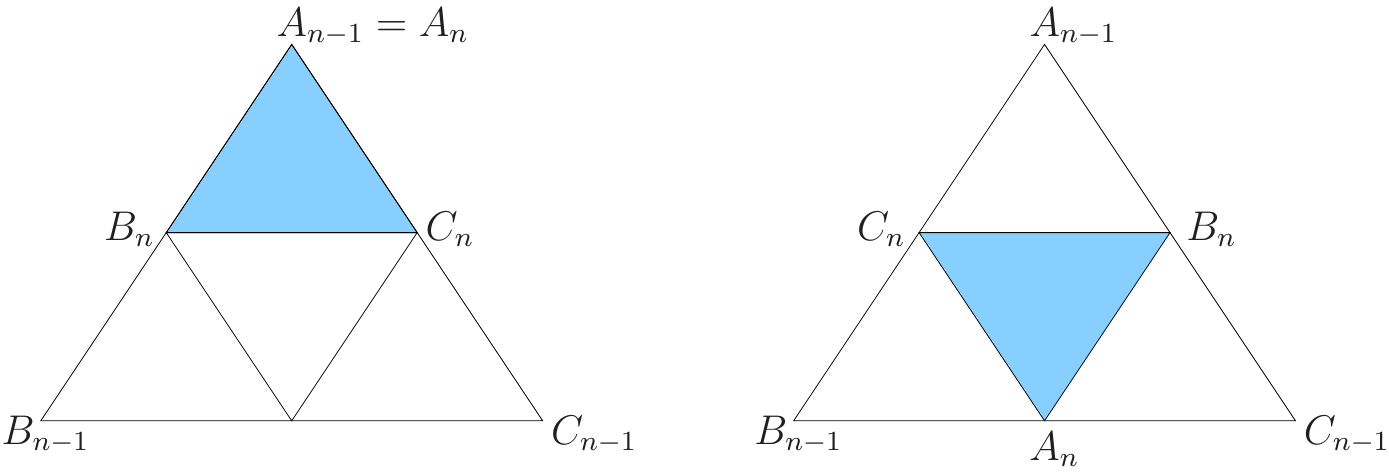}}\\
\begin{minipage}{0.9\textwidth}
\caption{\label{fig:iteration_triangle}
{\bf The inductive definition.} The shadowed triangle is $\triangle_n$. }
\end{minipage}
\end{figure}
We use $A_n$ to denote the angle $\angle B_nA_nC_n$, and $a_n$ the edge $B_nC_n$ (or its length). We define $B_n$, $C_n$, $b_n$ and $c_n$ similarly. Let $S_n$ be the (hyperbolic) area of $\triangle_n$. We have $A_n+B_n+C_n=1-S_n$.

The total angle will converge to $\pi$ since the area will converge to zero. It is then prone to think that each angle will increase and share the angle defect. However, this is not the case. For example, for an isosceles triangle $\triangle_*$ with edges lengths $4$, $4$ and $7$, the apex angle will increase while the bottom angles will decrease for $f_M$. So it is not trivial that the sequence will automatically converge to a nondegenerate Euclidean triangle. The distance of $\triangle_*$ from the fixed point $\left(\displaystyle{\frac{\pi}3,\frac\pi3,\frac\pi3}\right)$ will increase. Hence $f_M$ is NOT a contraction map.

The following lemma estimate the edge length changing under the IFS.
\begin{lemma}\label{lemma:edge_estimate}
For the sequence $\left\{\triangle_n\right\}_{n=0}^\infty$, we have
$$\sinh\frac{a_{n+1}}2<\frac12\sinh\frac{a_{n}}2.$$
Under the assumptions $\max\left\{\displaystyle{\sinh\frac{a_0}2}, \displaystyle{\sinh\frac{b_0}2}, \displaystyle{\sinh\frac{c_0}2}\right\}<\sigma$, we have
$$\left(\frac1{e\sqrt{e}}\right)^\sigma\cdot\frac{1}{2^n}\sinh\frac{a_0}2<\sinh\frac{a_{n}}2<\frac1{2^n}\sinh\frac{a_{0}}2.$$
\end{lemma}

\begin{proof} In the triangle $\triangle_n$, let $\ell$ be the line passing through the midpoints $C_{n+1}$ on $A_nB_n$ and $B_{n+1}$ on $A_nC_n$. Draw lines through $B_n$ and the midpoint $A_{n+1}$ of $B_nC_n$ to $\ell$. Let $E$ and $G$ be the intersection points respectively. Then we get the \emph{Lambert quadrilateral} as in Figure \ref{fig:lambert_quadrilateral}.
\begin{figure}[htbp]
\begin{center}
\includegraphics{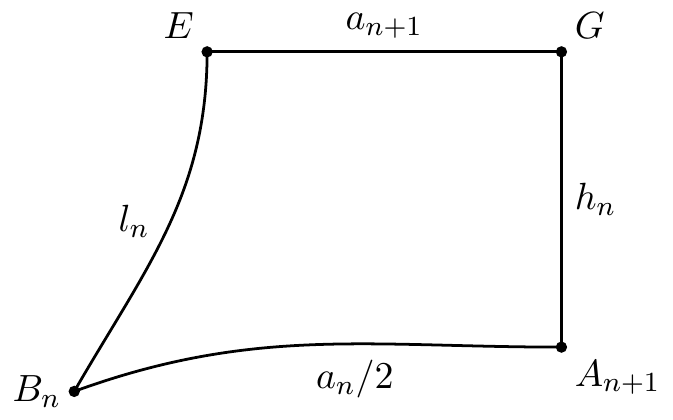}
\end{center}
\begin{minipage}{0.9\textwidth}
\caption{\label{fig:lambert_quadrilateral}
{\bf  The Lambert quadrilateral.}}
\end{minipage}
\end{figure}
The segment $A_{n+1}G$ is also perpendicular to $B_nA_{n+1}$. By congruence of triangles, the length of $EG$ is the same as the length of the mid-line $B_{n+1}C_{n+1}$. We have
 $$\sinh\left(\frac{a_n}2\right)=\sinh(a_{n+1})\cosh(l_n)=2\sinh\frac{a_{n+1}}2\cosh\frac{a_{n+1}}2\cosh{l_n}$$
Therefore
$$\sinh\left(\frac{a_{n+1}}2\right)\left/\sinh\left(\frac{a_n}2\right)\right.=\frac{1}{2\cosh(a_{n+1}/2)\cosh(l_n)}<\frac12$$
When $\displaystyle{a_{n+1}=\frac12a_n}$, the inequality comes from the convexity of the function $\sinh(x)$.

Therefore, we may suppose $\displaystyle{\sinh\frac{a_0}2<1}$, $\displaystyle{\sinh\frac{b_0}2<1}$ and $\displaystyle{\sinh\frac{c_0}2<1}$. Then we have
$$\sinh\frac{a_n}2<\frac{1}{2^n},\quad\sinh\frac{b_n}2<\frac1{2^n},\quad\sinh\frac{c_n}2<\frac1{2^n}$$
When $a_n$ joins the midpoint, the Lambert quadrilateral tells that $l_n<\displaystyle{\frac{b_n}2}$, and hence we have
 we have
$$\sinh^2\frac{a_n}2=\frac{1}{4\cosh^2(a_n/2)\cosh^2(l_{n-1})}\cdot\sinh^2\frac{a_{n-1}}2>\frac{1}{4\cosh^2(a_n/2)\cosh^2(b_{n-1}/2)}\cdot\sinh^2\frac{a_{n-1}}2$$
When $a_n=\displaystyle{\frac12a_{n-1}}$, we have
$$\sinh^2\frac{a_n}2=\frac{1}{4\cosh^2(a_n/2)}\cdot\sinh^2\frac{a_{n-1}}2>\frac{1}{4\cosh^2(a_n/2)\cosh^2(b_{n-1}/2)}\cdot\sinh^2\frac{a_{n-1}}2$$
In all cases, we have
{\allowdisplaybreaks\begin{align*}
\sinh^2\frac{a_n}2>&\frac{1}{4\cosh^2(a_n/2)\cosh^2(b_{n-1}/2)}\cdot\sinh^2\frac{a_{n-1}}2\\
>&\cdots>\frac{1}{4^n}\sinh^2\frac{a_0}2\left/\left(\prod_{k=1}^n\big(1+\sinh^2(a_k/2)\big)\big(1+\sinh^2(b_{k-1}/2)\big)\right)\right.\\
>&\frac{1}{4^n}\sinh^2\frac{a_0}2\left/\exp\left(\sum_{k=1}^\infty\ln\big(1+\sinh^2(a_k/2)\big)+\sum_{k=0}^\infty\ln\big(1+\sinh^2(b_{k-1}/2)\big)\right)\right.\\
>&\frac{1}{4^n}\sinh^2\frac{a_0}2\left/\exp\left(\sum_{k=1}^\infty\ln\big(1+\frac1{2^k}\big)+\sum_{k=0}^\infty\ln\big(1+\frac{1}{2^k}\big)\right)\right.\\
>&\frac{1}{4^n}\sinh^2\frac{a_0}2\left/\exp\left(\sum_{k=1}^\infty\frac1{2^k}+\sum_{k=0}^\infty\frac{1}{2^k}\right)\right.=\frac1{4^ne^3}\sinh^2\frac{a_0}2\\
\end{align*}}
So
$$\sinh\frac{a_n}2>\frac{1}{e\sqrt{e}}\cdot\frac{1}{2^n}\sinh\frac{a_0}2.$$
\end{proof}

\begin{prop}\label{prop:converging_nondegenerate}
For any hyperbolic triangle $\triangle_0$, the sequence $\left\{\triangle_n\right\}_{n=0}^\infty$ converges to a nondegenerate Euclidean triangle.
\end{prop}

\begin{proof} Let $l_{a_n}$ be the length of the lambert qudrilateral on $a_n$, as in figure 2. Then we have
\begin{equation}\label{eqn:lambert_identity}
\sinh\frac{a_n}2=\sinh\frac{a_{n+1}}2\cosh l_{a_n}
\end{equation}

Now we compare $\sin\beta$ with $\sin A_n$. By the law of sines, we have
$$\sin\beta=\frac{\displaystyle{\sinh\frac{a_n}{2}\sin B_n}}{\sinh b_{n+1}},\quad \sin A_n=\frac{\sinh a_n\sin B_n}{\sinh b_n}$$
Hence
$$\frac{\sin\beta}{\sin A_n}=\frac{\displaystyle{\sinh(a_n/2)\sinh b_n}}{\sinh b_{n+1}\sinh a_n}=\frac{\displaystyle{\sinh(a_n/2)\sinh b_n\cosh l_{b_n}}}{\displaystyle{\sinh\frac {b_n}2\sinh a_n}}=\frac{\cosh(l_{b_n})}{\cosh(a_n/2)}\cosh(b_n/2)$$
In the corresponding right-angle triangle, we get $\displaystyle{l_{b_n}<\frac{a_n}2}$, hence we obtain
\begin{equation}
\frac{\cosh(b_n/2)}{\cosh(a_n/2)}<\frac{\sin\beta}{\sin A_n}<\cosh(b_n/2)
\end{equation}
Similarly, we get
\begin{equation}
\frac{\cosh(c_n/2)}{\cosh(a_n/2)}<\frac{\sin\gamma}{\sin A_n}<\cosh(c_n/2)
\end{equation}

Now we compare $\sin\alpha$ and $\sin A_n$. By the Cagnoli formula, we have
\begin{equation}
\sin\frac{S_n}2=\frac{\sinh(b_{n}/2)\sinh(c_n/2)\sin A_n}{\cosh(a_n/2)}
\end{equation}
By the Keogh formula, $S_n$ can also be computed by $b_{n+1},c_{n+1}$ and $\alpha$ as follows
\begin{equation}
\sin\frac{S_n}2=\sinh(b_{n+1})\sinh(c_{n+1})\sin(\alpha)
\end{equation}
So we get
\begin{align*}
\frac{\sin\alpha}{\sin A_n}=& \frac{\sinh(b_{n}/2)\sinh(c_n/2)}{\sinh(b_{n+1})\sinh(c_{n+1})}\cdot\frac{1}{\cosh(a_n/2)}\\
=& \frac{\cosh(l_{b_n})\cosh(l_{c_n})}{\cosh(a_n/2)}\quad (\text{Apply }(\ref{eqn:lambert_identity}))
\end{align*}
Apply $l_{b_n}<\frac{c_n}2$ and $l_{c_n}<\frac{b_n}2$ to get the inequality
$$\frac1{\cosh(a_n/2)}<\frac{\sin \alpha}{\sin A_n}<\cosh(b_n/2)\cosh(c_n/2)$$

So in all cases, we get
\begin{equation}
\frac{1}{\cosh(a_n/2)}<\frac{\sin A_{n+1}}{\sin A_n}<\cosh(b_n/2)\cosh(c_n/2)
\end{equation}

Let $\rho_n=\ln\sin A_n$, then we have
$$-\sum_{i=n}^{n+k-1}\ln\cosh\left(\frac{a_i}2\right)<\rho_{n+k}-\rho_n<\sum_{i=n}^{n+k-1}\left(\ln \cosh\left(\frac{b_i}2\right)+\ln\cosh\left(\frac{c_i}2\right)\right)$$
Hence
\begin{align*}
\abs{\rho_{n+k}-\rho_n}<&\sum_{i=n}^{n+k-1}\left(\ln \cosh\frac{a_i}2+\ln \cosh \frac{b_i}2 +\ln\cosh\frac{c_i}2\right)\\
=&\frac12\sum_{i=n}^{n+k-1}\left(\ln\left(1+\sinh^2\frac{a_i}2\right)+\ln\left(1+\sinh^2\frac{b_i}2\right)+\ln\left(1+\sinh^2\frac{c_i}2\right)\right)\\
<&\frac12\sum_{i=n}^{n+k-1}\left(\sinh^2\frac{a_i}2+\sinh^2\frac{b_i}2+\sinh^2\frac{c_i}2\right)\\
<&\frac12\sum_{i=n}^{\infty}\frac{1}{2^i}\left(\sinh^2\frac{a_0}2+\sinh^2\frac{b_0}2+\sinh^2\frac{c_0}2\right)\\
=&\frac{1}{2^n}\left(\sinh^2\frac{a_0}2+\sinh^2\frac{b_0}2+\sinh^2\frac{c_0}2\right)
\end{align*}
Therefore $\left\{\ln\sin A_n\right\}$ is a Cauchy sequence and converges. Similarly, $\{\ln\sinh B_n\}$ and $\{\ln\sin C_n\}$ converge. The proposition follows.

\end{proof}

For iteration of $f_M$, we compute the areas of the triangles as follows

\begin{prop} Let $\triangle_n=f^{n}_M(\triangle_0)$. Suppose that $\displaystyle{\sinh\frac{a_0}2<1}$, $\displaystyle{\sinh\frac{b_0}2<1}$ and $\displaystyle{\sinh\frac{c_0}2<1}$, then
$$\frac1{\sqrt{e}}\cdot\frac{1}{4^n}\cdot\sinh\frac{S_0}2\leq \sinh\frac{S_n}2\leq \frac{1}{4^n}\sinh\frac{S_0}2,\quad\forall\ n.$$
\end{prop}

\begin{proof} For any $n$, let 
$$x=2\cosh a_{n},\quad y=2\cosh b_{n},\quad z=2\cosh c_{n}$$
Then by the trace identity, we have
$$\left(2\cos\frac{S_{n-1}}2\right)^2=x^2+y^2+z^2-xyz$$
and
$$\left(\cos\frac{S_{n}}2\right)^2=\frac{(x+y+z+2)^2}{(x+2)(y+2)(z+2)}$$
Therefore we have
{\allowdisplaybreaks\begin{align*}
\left(\sin\frac{S_{n}}2\right)^2=&1-\left(\cos\frac{S_{n}}2\right)^2=1-\frac{(x+y+z+2)^2}{(x+2)(y+2)(z+2)}\\
=&\frac{(x+2)(y+2)(z+2)-(x+y+z+2)^2}{(x+2)(y+2)(z+2)}\\
=&\frac{4-(x^2+y^2+z^2-xyz)}{(x+2)(y+2)(z+2)}=\frac{4-4\left(\cos\displaystyle{\frac{S_{n-1}}2}\right)^2}{(x+2)(y+2)(z+2)}\\
=&\frac{4}{(2\cosh a_{n}+2)(2\cosh b_{n}+2)(2\cosh c_{n}+2)}\left(\sin\frac{S_{n-1}}2\right)^2\\
=&\frac{4}{(4\sinh^2 \frac{a_{n}}2+4)(4\sinh^2 \frac{b_{n}}2+4)(4\sinh^2 \frac{c_{n}}2+4)}\left(\sin\frac{S_{n-1}}2\right)^2\\
=&\left(\frac14\sin\frac{S_{n-1}}2\right)^2\left/\left[(1+\sinh^2 \frac{a_{n}}2)(1+\sinh^2 \frac{b_{n}}2)(1+\sinh^2 \frac{c_{n}}2)\right]\right.
\end{align*}}
Without loss of generality, we may assume $\sinh\displaystyle{\frac{a_0}2}<1$. By Lemma, we have $\sinh\displaystyle{\frac{a_{n}}2}<\frac12 \sinh\frac{a_{n-1}}2$, then $\displaystyle{\sinh\frac{a_n}2<\frac1{2^n}}$, we get
\begin{align*}
\left(\sin\frac{S_{n}}2\right)^2\geq&\left(\frac14\sin\frac{S_{n-1}}2\right)^2\left/\left(1+\frac1{4^n}\right)^3\right.\geq\cdots\geq\frac{1}{16^n}\left(\sin\frac{S_0}2\right)^2\left/\prod_{i=1}^n\left(1+\frac1{4^i}\right)^3\right.\\
=&\frac{1}{16^n}\left(\sin\frac{S_0}2\right)^2\left/\exp\left\{3\sum_{i=1}^n\ln\left(1+\frac1{4^i}\right)\right\}\right.\\
\geq&\frac{1}{16^n}\left(\sin\frac{S_0}2\right)^2\left/\exp\left\{3\sum_{i=1}^\infty\frac1{4^i}\right\}\right.=\frac1{16^n}\cdot\frac1e\cdot\left(\sin\frac{S_0}2\right)^2
\end{align*}
\end{proof}

\begin{cor}  For the sequence $\left\{\triangle_n=f_M^n(\triangle_0)\right\}$, let $S_n$ be the area of $\triangle_n$, then the sequence $\displaystyle{\left\{\frac{\sin(S_n/2)}{\sin(S_0/2)}\cdot4^n\right\}}$ converges and its limit belongs to $\displaystyle{\left(\frac{1}{\sqrt{e}},\sqrt{e}\right)}$.
\end{cor}

So the sine of the area of the middle triangle is shrinking to approximately $\frac14$ the sine of area of the previous triangle, with accumulated error of ratio less than $\sqrt{e}$.

\section{Continuity of the limit functions}

Let $\calt_E$ be the Techm\"uller space of nontrivial Euclidean triangles modulo similarities and $\calt_H$ be the Techm\"uller space of nontrivial hyperbolic triangles. $\calt_E$ is naturally identified with the interior of the triangle in $\bbr^3$ with vertices $(1,0,0)$, $(0,1,0)$ and $(0,0,1)$. $\calt_H$ is naturally identified with the interior of the tetrahedron in $\bbr^3$ with vertices $(0,0,0)$, $(1,0,0)$, $(0,1,0)$ and $(0,0,1)$. $\calt_E$ and $\calt_H$ inherit metrics as subspaces of $\bbr^3$. We have $\calt=\calt_E\cup\calt_H$.

Let $\cals$ be the set of infinite sequences in four letters $A$, $B$, $C$ and $M$. A sequence $\fraks$ is called {\it rational} if exactly one of the three letters $A$, $B$ and $C$ appears infinite times, and {\it irrational} otherwise. Fix a Eulidean or hyperbolic triangle $\triangle=\triangle_{ABC}$. A sequence in $s\in\cals$ defines nested triangles $\{\triangle_n\}$ in $\triangle$ via $f_A$, $f_B$, $f_C$ and $f_M$. The nested triangles $\{\triangle_n\}$ have a unique intersection point, denoted by $\phi(s)$. It is not hard to see that $\phi$ defines a surjective map from $\cals$ to $\triangle$. Endow $\cals$ with the smallest topology $\calt$ such that $\phi$ is continuous. The topology does not depend on the choice of the Euclidean triangle $\triangle$.

An easy investigation of $\phi$ gives the following 

\begin{prop}\label{prop:sequence_map}
For any two distinct sequences $\fraks$ and $\frakt$, we have $\phi(\fraks)=\phi(\frakt)$ if and only if there exists
\begin{itemize}
\item a finite sequence $\tau_1,\cdots,\tau_{n}$ for some $n\geqslant0$,
\item a permutation $\sigma:\{A,B,C\}\rightarrow\{A,B,C\}$, and
\item a finite sequence $\zeta:\{1,\cdots,m\}\rightarrow\{x,y\}$ in two indeterminants $x$ and $y$ with $m\geqslant0$, 
\end{itemize}
such that $\fraks$ and $\frakt$ are of the following six forms

\begin{tabular}{c c c c c c c c c c c c c c}
$(1)$&$\tau_1$&$\tau_2$&$\cdots$&$\tau_{n}$&$\sigma(A)$&$\alpha_1$&$\alpha_2$&$\cdots$&$\alpha_m$&$M$&$\sigma(A)$&$\sigma(A)$&$\cdots$\\
$(2)$&$\tau_1$&$\tau_2$&$\cdots$&$\tau_{n}$&$\sigma(A)$&$\alpha_1$&$\alpha_2$&$\cdots$&$\alpha_m$&$\sigma(B)$&$\sigma(C)$&$\sigma(C)$&$\cdots$\\
$(3)$&$\tau_1$&$\tau_2$&$\cdots$&$\tau_{n}$&$\sigma(A)$&$\alpha_1$&$\alpha_2$&$\cdots$&$\alpha_m$&$\sigma(C)$&$\sigma(B)$&$\sigma(B)$&$\cdots$\\
$(4)$&$\tau_1$&$\tau_2$&$\cdots$&$\tau_{n}$&$M$&$\beta_1$&$\beta_2$&$\cdots$&$\beta_m$&$M$&$\sigma(A)$&$\sigma(A)$&$\cdots$\\
$(5)$&$\tau_1$&$\tau_2$&$\cdots$&$\tau_{n}$&$M$&$\beta_1$&$\beta_2$&$\cdots$&$\beta_m$&$\sigma(B)$&$\sigma(C)$&$\sigma(C)$&$\cdots$\\
$(6)$&$\tau_1$&$\tau_2$&$\cdots$&$\tau_{n}$&$M$&$\beta_1$&$\beta_2$&$\cdots$&$\beta_m$&$\sigma(C)$&$\sigma(B)$&$\sigma(B)$&$\cdots$
\end{tabular}

\noindent where $\alpha_i=\zeta_i(\sigma(B),\sigma(C))$ and $\beta_i=\zeta_i(\sigma(C),\sigma(B))$.
\end{prop}

Two distinct sequences with the same image under $\phi$ are rational. An irrational sequence does not have the same image under $\phi$ with another sequence. Proposition \ref{prop:sequence_map} implies that the space $(\cals, \calt)$ is not Hausdorff.

Given a sequence $\fraks\in\cals$ and a hyperbolic triangle $\triangle_0=\triangle_{A_0B_0C_0}\in\calt_H$. Inductively define $\triangle_n=f_{s_n}(\triangle_{n-1})$. By Proposition \ref{prop:converging_nondegenerate}, the sequence $\{\triangle_n\}$ converge to a nondegenerate Euclidean triangle. Hence we have a well-defined map $\Phi:\cals\times \calt_H\rightarrow \calt_E$.

\begin{prop}\label{prop:continuity}
The map $\Phi:\cals\times \calt_H\rightarrow \calt_E$ satisfies the following continuity properties
\begin{enumerate}
\item For any $\fraks\in\cals$, $\Phi_\fraks=\Phi(\fraks,\cdot):\calt_H\rightarrow \calt_E$ is continous and surjective.
\item For any $\triangle\in \calt_H$, $\Phi_\triangle=\Phi(\cdot,\triangle):\cals\rightarrow \calt_E$ is continous at irrational points.
\end{enumerate}
\end{prop}

\begin{proof}

The functions $f_A$, $f_B$, $f_C$ and $f_M$ are smooth. For the continuity, it suffices to show that the series $\displaystyle{\sum_{n=1}^\infty(\ln\sin A_n-\ln\sin A_{n-1})}$ is uniformly convergent, which directly follows from the last inequality in the proof of Proposition \ref{prop:converging_nondegenerate}.

Let $p:\calt_H\rightarrow\calt_E$ be the projection
$$(x,y,z)\longmapsto \left(\frac{x}{x+y+z}\cdot\pi,\frac{z}{x+y+z}\cdot\pi,\frac{z}{x+y+z}\cdot\pi\right).$$

If $\Phi_\fraks$ is not surjective, take $v\in\calt_E\setminus\phi_\fraks(\calt_H)$. There exists some $\varepsilon>0$ such that $\calt_E$ contains the circle $C_1$ with center $v$ and radius $\varepsilon$.

By the proof of Proposition \ref{prop:converging_nondegenerate}, there exists some sufficiently small $s>0$ such that
$$\abs{\Phi_\fraks(x)-p(x)}< \frac{\varepsilon}2$$
for any $x$ in the circle $C_2$, which is the intersection of $p^{-1}(C_1)$ and $x+y+z=\pi-s$. Then the singular loop $\Phi_\fraks(C_2)$ is free homotopic to $C_1$ in $\calt_E\setminus\{v\}$ by straight-line homotopy, and hence not null homotopic. However, $C_2$ bounds a disk in $x+y+z=\pi-s$ whose image gives a null homotopy of $\Phi_\fraks(C_2)$. The contradiction implies the surjectivity of $\Phi_s$.

Given a hyperbolic triangle $\triangle=\triangle_{ABC}$. Use $\triangle$ to define the topology on $\cals$. Let $\fraks=(s_1,s_2,\cdots)$ be an irrational sequence. Let $\{\triangle_n\}$ be the sequence to define $\Phi(\fraks,\triangle)$. For any $n$, the set $U_n$ of sequences starting with $(s_1,s_2,\cdots,s_n)$ is a neighborhood of $\fraks$. Given any $\varepsilon\in(0,1)$, let 
$$\mu=\min\left\{\ln\left(\frac\varepsilon{14\pi}+1\right), \ln\left(\frac{1}{1-\varepsilon/(14\pi)}\right)\right\}$$
By Proposition \ref{prop:converging_nondegenerate}, there exist $N$ such that for any $\frakt\in U_N$, the ratio of the angles of $\Phi(\frakt,\triangle)$ and the angles $(A_N,B_N,C_N)$ of $\triangle_N$ is within $(e^{-\mu}, e^\mu)$. For any $\frakt\in U_N$, suppose $\Phi_\triangle(\frakt)=(\alpha,\beta,\gamma)$ and $\Phi_\triangle(\fraks)=(\alpha_0,\beta_0,\gamma_0)$, then
\begin{align*}
d(\Phi_\triangle(\frakt),\Phi_\triangle(\fraks))\leq&\abs{\alpha-\alpha_0}+\abs{\beta-\beta_0}+\abs{\gamma-\gamma_0}\\
\leq& \abs{\alpha-A_N} +\abs{A_N-\alpha_0}+\abs{\beta-B_N}+\abs{B_N-\beta_0}+\abs{\gamma-C_N}+\abs{C_N-\gamma_0}\\
\leq& 2(A_N+B_N+C_N)\cdot \big((1-e^{-\mu})+(e^\mu-1)\big)\\
\leq& 6\pi\cdot \left(\frac\varepsilon{14\pi}+\frac\varepsilon{14\pi}\right)=\frac67\cdot\varepsilon<\varepsilon.
\end{align*}
Hence $\Phi_\triangle$ is continuous at $\triangle$.
\end{proof}

We remark that in generral $\Phi_\triangle$ can be discontinuous at an rational sequence.

\begin{proof}[Proof of Theorem \ref{thm:main}]
The theorem evidently follows from Proposition \ref{prop:converging_nondegenerate} and Proposition \ref{prop:continuity}.
\end{proof}

\subsubsection*{Acknowledgement} The first named author is partially supported by NSFC 11425102.

\end{document}